\documentclass[reqno]{amsart}
\usepackage[utf8x]{inputenc}  

\usepackage{setspace}
\usepackage[left=2.8cm, right=2.8cm, bottom=3cm]{geometry}                 

\usepackage{graphicx} 
\usepackage{pgf,tikz}
\usetikzlibrary{arrows}
\usepackage{amssymb}
\usepackage{pdfsync}
\usepackage{mathrsfs}
\usepackage{hyperref} 
\usepackage{verbatim} 
\usepackage{epstopdf}
\usepackage{bbm} 
\usepackage[colorinlistoftodos,prependcaption,textsize=tiny]{todonotes}
\usepackage{xargs}

\def\R{\mathbb{R}}

\def\N{\mathbb{N}}

\def\C{\mathbb{C}}

 \newcommand{\ds}{\text{\rm d}s}
 \newcommand{\dt}{\text{\rm d}t}

 \newcommand{\dx}{\text{\rm d}x}
 
 \newcommand{\dxi}{\text{\rm d}\xi}
  \newcommand{\deta}{\text{\rm d}\eta}
 \newcommand{\dnu}{\text{\rm d}\nu}
  \newcommand{\domega}{\text{\rm d}\omega}

\newcommandx{\emanuel}[2][1=]{\todo[linecolor=green,backgroundcolor=green!25,bordercolor=black,#1]{#2}}

\newcommandx{\diogo}[2][1=]{\todo[linecolor=orange,backgroundcolor=orange!25,bordercolor=orange,#1]{#2}}

\newcommandx{\mateus}[2][1=]{\todo[linecolor=blue,backgroundcolor=blue!25,bordercolor=blue,#1]{#2}}

\newcommandx{\danger}[2][1=]{\todo[linecolor=red,backgroundcolor=red!25,bordercolor=blue,#1]{#2}}

\newtheorem{theorem}{Theorem}

\newtheorem{proposition}[theorem]{Proposition}

\makeatletter
\DeclareFontFamily{U}{tipa}{}
\DeclareFontShape{U}{tipa}{m}{n}{<->tipa10}{}
\newcommand{\arc@char}{{\usefont{U}{tipa}{m}{n}\symbol{62}}}%

\newcommand{\arc}[1]{\mathpalette\arc@arc{#1}}

\newcommand{\arc@arc}[2]{%
  \sbox0{$\m@th#1#2$}%
  \vbox{
    \hbox{\resizebox{\wd0}{\height}{\arc@char}}
    \nointerlineskip
    \box0
  }%
}
\makeatother

\numberwithin{equation}{section}

\allowdisplaybreaks

\newcommand{\intav}[1]{\mathchoice {\mathop{\vrule width 6pt height 3 pt depth  -2.5pt
\kern -8pt \intop}\nolimits_{\kern -6pt#1}} {\mathop{\vrule width
5pt height 3  pt depth -2.6pt \kern -6pt \intop}\nolimits_{#1}}
{\mathop{\vrule width 5pt height 3 pt depth -2.6pt \kern -6pt
\intop}\nolimits_{#1}} {\mathop{\vrule width 5pt height 3 pt depth
-2.6pt \kern -6pt \intop}\nolimits_{#1}}}

\newcommand{\intavl}[1]{\mathchoice {\mathop{\vrule width 6pt height 3 pt depth  -2.5pt
\kern -8pt \intop}\limits_{\kern -6pt#1}} {\mathop{\vrule width 5pt
height 3  pt depth -2.6pt \kern -6pt \intop}\nolimits_{#1}}
{\mathop{\vrule width 5pt height 3 pt depth -2.6pt \kern -6pt
\intop}\nolimits_{#1}} {\mathop{\vrule width 5pt height 3 pt depth
-2.6pt \kern -6pt \intop}\nolimits_{#1}}}

\title[Gaussians never extremize]{Gaussians never extremize Strichartz inequalities for hyperbolic paraboloids}

\author[Carneiro]{Emanuel Carneiro}
\author[Oliveira]{Lucas Oliveira}
\author[Sousa]{Mateus Sousa}

\address{
ICTP - The Abdus Salam International Centre for Theoretical Physics\\
Strada Costiera, 11, I - 34151, Trieste, Italy.}

\address{IMPA - Instituto de Matem\'{a}tica Pura e Aplicada\\
Rio de Janeiro - RJ, Brazil, 22460-320.}

\email{carneiro@ictp.it}
\email{carneiro@impa.br}

\address{Department of Pure and Applied Mathematics, Universidade Federal do Rio Grande do Sul, Porto Alegre, RS, Brazil 91509-900.}
\email{lucas.oliveira@ufrgs.br}

\address{Mathematisches Institut der Universit\"{a}t M\"{u}nchen, Theresienstr. 39, D-80333, M\"{u}nchen, Germany.}
\email{sousa@math.lmu.de}

\date{\today}                                           
\begin{document}

\begin{abstract} 
For $\xi = (\xi_1, \xi_2, \ldots, \xi_d) \in \R^d$ let $Q(\xi) := \sum_{j=1}^d \sigma_j \xi_j^2$ be a quadratic form with signs $\sigma_j \in \{\pm1\}$ not all equal. Let $S \subset \R^{d+1}$ be the hyperbolic paraboloid given by $S = \big\{(\xi, \tau) \in \R^{d}\times \R \ : \ \tau = Q(\xi)\big\}$. In this note we prove that Gaussians never extremize an $L^p(\R^d) \to L^{q}(\R^{d+1})$ Fourier extension inequality associated to this surface.
\end{abstract}

\maketitle 

\section{Introduction}
Let $\xi = (\xi_1, \xi_2, \ldots, \xi_d) \in \R^d$ and consider a quadratic form $Q(\xi) := \sum_{j=1}^d \sigma_j \xi_j^2$, where each sign $\sigma_j \in \{\pm1\}$ is given. Associated to this quadratic form, let $S \subset \R^{d+1}$ be the surface given by $S = \big\{(\xi, \tau) \in \R^{d}\times \R \ : \ \tau = Q(\xi)\big\}$. If all the $\sigma_j$'s are equal we call such a surface a {\it paraboloid}, and in case we have different signatures we call such a surface a {\it hyperbolic paraboloid}. For $f:\R^d \to \C$ and $(x,t) \in \R^d \times \R$ we consider the Fourier extension operator associated to $S$:
\begin{equation}\label{operator}
Tf(x, t) = \int_{\R^d} e^{i x\cdot \xi + i t Q(\xi)} f(\xi)\,\dxi.
\end{equation}
Here we are concerned with estimates of the type 
\begin{equation}\label{20191004_12:11am}
\|Tf\|_{L^q(\R^{d+1})} \leq C \,\|f\|_{L^p(\R^{d})},
\end{equation}
which are fundamental in harmonic analysis. Dilation invariance tells us that such a global estimate can only hold if $q = (d+2)p'/d$ and inequality \eqref{20191004_12:11am} obviously holds true for $p=1$. The celebrated work of Strichartz \cite{St77} establishes it for $p=2$ (and hence for $1\leq p\leq 2$) but the full range for which \eqref{20191004_12:11am} holds is still unknown. The {\it restricton conjecture} predicts its validity for $1 \leq p < 2(d+1)/d$, see for instance \cite{Var}.

\smallskip 

When \eqref{20191004_12:11am} is known to hold, the search for its sharp form and extremizers is a deep and beautiful problem that has only been solved in a couple of cases, all when $S$ is the usual {\it paraboloid}. In this situation, when $p=2$, it is conjectured that Gaussians should extremize \eqref{20191004_12:11am}. This was established by Foschi  \cite{Fo07} in dimensions $d=1$ and $d=2$, and it is an important open problem in higher dimensions. A crucial point for Foschi's argument in low dimensions is the fact that the exponent $q$ is even, which yields a certain convolution structure to the problem. A natural question of interest at the time was if Gaussians could be extremizers in the general $L^p \to L^q$ setting for {\it paraboloids}, but this was disproved by Christ and Quilodr\'{a}n \cite{CQ}, who showed that  Gaussians are not even critical points of the inequality when $1 < p < 2(d+1)/d$ and $p \neq 2$.

\smallskip The purpose of this note is to establish a negative result analogous to that of Christ and Quilodr\'{a}n \cite{CQ} in all the cases of {\it hyperbolic paraboloids}.

\begin{theorem}\label{Thm1}
Let $S$ be a hyperbolic paraboloid and $1 < p < 2(d+1)/d$. Then Gaussians are not critical points of \eqref{20191004_12:11am} (in case the inequality holds).
\end{theorem}

A particularly attractive case of \eqref{20191004_12:11am} in the case of {\it hyperbolic paraboloids} is when $d=2$ and $S = \{(\xi_1,\xi_2, \tau) \in \R^3 \ : \ \tau = \xi_1^2 - \xi_2^2\}$. This is the only case when the exponent $q$ is even ($q=4$). The corresponding $L^2(\R^2)\to L^4(\R^3)$ Fourier extension estimate for this saddle surface and its connections with hyperbolic Schr\"{o}dinger equations have been studied by Rogers and Vargas \cite{RV}, and a profile decomposition for solutions of such equations was established by Dodson, Marzuola, Pausader and Spirn \cite{DMPS}. In the last section of the paper we present a brief discussion on some subtle points of this Fourier extension estimate, in particular addressing the difficulties of applying previous approaches by Foschi \cite{Fo07} and Hundertmark and Zharnitsky \cite{HZ06} for the {\it paraboloid}.

\section{Proof of Theorem \ref{Thm1}}

We divide the proof in several simple steps.

\subsection{Computing the Fourier extension a Gaussian} Throughout the paper we fix the $L^2$-normalized Gaussian
\begin{equation}\label{Gaussian}
g(\xi) = e^{-\frac{\pi |\xi|^2}{2}}.
\end{equation}
Letting $x = (x_1, x_2, \ldots, x_d)$, a routine computation yields
\begin{equation}\label{20191007_14:11pm}
Tg(x, t) = \prod_{k=1}^d \left(\frac12 - \frac{i \sigma_k t}{\pi}\right)^{-1/2} \, e^{-\frac{x_k^2}{4\pi\left(\frac12 -\frac{i \sigma_k t}{\pi}\right)}}.
\end{equation}
Hence it follows that
\begin{equation}\label{20191007_14:12pm}
|Tg(x, t)| =  \left(\frac{1}{4} + \frac{t^2}{\pi^2}\right)^{-d/4}\,e^{-\frac{\pi |x|^2}{2(\pi^2 + 4t^2)}}.
\end{equation}

\subsection{Computing the first variation} Let $\phi \in C^{\infty}_c(\R^d)$. For $\varepsilon \in \R$ small we define
\begin{equation*}
\Psi(\varepsilon) = \frac{\|T(g + \varepsilon \phi)\|_{L^q(\R^{d+1})}}{\|g + \varepsilon \phi\|_{L^p(\R^{d})}}.
\end{equation*}
We say that $g$ is a critical point of \eqref{20191004_12:11am} if $\Psi'(0) = 0$ for all such $\phi$. If we further assume the orthogonality condition 
$$\int_{\R^d} g(\xi)^{p-1} \, \phi(\xi)\,\dxi = 0\,,$$
a routine computation yields
\begin{equation*}
\Psi'(0) = \frac{\|Tg\|_{q}^{1-q}}{\|g\|_{p}}\, {\rm Re} \left(\int_{\R^{d+1}} |Tg|^{q-2}\,  \,Tg \,\, \overline{T\phi} \right).
\end{equation*}
Using \eqref{operator}, \eqref{20191007_14:11pm}, \eqref{20191007_14:12pm} and Fubini's theorem (in the following order: first change $x$ and $\xi$, then integrate in $x$, then change $\xi$ and $t$; note that we have absolute integrability of the integrands in each step) we observe that 
\begin{align*}
\begin{split}
& {\rm Re} \int_{\R^{d+1}} |Tg|^{q-2}\,  \,Tg \,\, \overline{T\phi} \\
& = {\rm Re} \int_{\R}  \left(\frac{1}{4} + \frac{t^2}{\pi^2}\right)^{-\frac{d(q-2)}{4}} \prod_{k=1}^d \left(\frac12 - \frac{i \sigma_k t}{\pi}\right)^{-\frac12} \int_{\R^{d}}e^{-\frac{\pi (q-2) |x|^2 }{2(\pi^2 + 4t^2)}} \left(\prod_{k=1}^d  e^{-\frac{x_k^2}{4\pi\left(\frac12 -\frac{i \sigma_k t}{\pi}\right)}}\right) \times \\
&  \ \ \ \ \ \ \ \ \ \ \ \ \ \ \ \  \ \ \ \ \ \ \ \ \ \ \ \ \ \ \ \  \ \ \ \ \ \ \ \ \ \ \ \ \ \ \ \  \times \left(\int_{\R^d} e^{-i x\cdot \xi - i t Q(\xi)} \,\overline{\phi(\xi)}\,\,\dxi\right) \dx\,\dt\\
& = {\rm Re} \int_{\R}  \left(\frac{1}{4} + \frac{t^2}{\pi^2}\right)^{-\frac{d(q-2)}{4}}\!\! \prod_{k=1}^d \left(\frac12 - \frac{i \sigma_k t}{\pi}\right)^{-\frac12}  \left(\frac{\pi^2 + 4t^2}{\frac{q-1}{2} + \frac{i \sigma_k t}{\pi}}\right)^{\frac12} \times\\
&  \ \ \ \ \ \ \ \ \ \ \ \ \ \ \ \  \ \ \ \ \ \ \ \ \ \ \ \ \ \ \ \  \ \ \ \ \ \ \ \ \ \ \ \ \ \ \ \  \times \int_{\R^{d}}\left(\prod_{k=1}^d e^{-\frac{\xi_k^2 (\pi^2 + 4t^2)}{4\pi \left(\frac{q-1}{2} + \frac{i \sigma_k t}{\pi}\right)}}\right) e^{- i t Q(\xi)} \,\overline{\phi(\xi)}\,\,\dxi \,\dt\\
& = {\rm Re} \int_{\R^d}  \overline{\phi(\xi)} \int_{\R^{d}}  \! \left(\frac{1}{4} + \frac{t^2}{\pi^2}\right)^{-\frac{d(q-2)}{4}}  \!\!\! \left( \prod_{k=1}^d \left(\frac12 - \frac{i \sigma_k t}{\pi}\right)^{-\frac12}\!\! \left(\frac{\pi^2 + 4t^2}{\frac{q-1}{2} + \frac{i \sigma_k t}{\pi}}\right)^{\frac12} \!  e^{-\frac{\xi_k^2 (\pi^2 + 4t^2)}{4\pi \left(\frac{q-1}{2} + \frac{i \sigma_k t}{\pi}\right)}} \right) \! e^{- i t Q(\xi)} \,\dt \,\dxi.
\end{split}
\end{align*}
We then find that $g$ is a critical point if and only if the following Euler-Lagrange equation is satisfied:
\begin{align}\label{20191007_15:18pm}
\int_{\R}\left(\frac{1}{4} + \frac{t^2}{\pi^2}\right)^{-\frac{d(q-2)}{4}}\!e^{- i t Q(\xi)}  \left(\prod_{k=1}^d  \left(\frac{\frac{1}{2} + \frac{i\sigma_kt}{\pi}}{\frac{q-1}{2} + \frac{i\sigma_kt}{\pi}} \right)^{\frac12} e^{-\frac{\xi_k^2 (\pi^2 + 4t^2)}{4\pi \left(\frac{q-1}{2} + \frac{i \sigma_k t}{\pi}\right)}}\right) \dt  =  \lambda \, g(\xi)^{p-1},
\end{align}
where $\lambda$ is a constant.

\subsection{Using the hiperbolicity} Let us now assume that the quadratic form $Q$ has $d^+$ signs $+1$ (say, $\sigma_1, \sigma_2, \ldots, \sigma_{d^+}$) and $d^{-}$ signs $-1$ (say, $\sigma_{d^+ +1}, \ldots, \sigma_{d}$), where $d^+ + d^- = d$. We may assume without loss of generality that $d^+ \geq d^- \geq 1$. We shall write a vector $\xi  = (\xi_1, \xi_2, \ldots, \xi_d) \in \R^d$ as $\xi = (\xi^+, \xi^-) \in \R^{d^+} \times \R^{d^-}$, with the understanding that $\xi^+ = (\xi_1, \ldots, \xi_{d^+})$ and $\xi^- = (\xi_{d^+ + 1}, \ldots, \xi_d)$. If we group the terms with the same signature, the Euler-Lagrange equation \eqref{20191007_15:18pm} reads
\begin{align*}
& \int_{\R}\left(\frac{1}{4} + \frac{t^2}{\pi^2}\right)^{-\frac{d(q-2)}{4}} \left(\frac{\frac{1}{2} + \frac{it}{\pi}}{\frac{q-1}{2} + \frac{it}{\pi}} \right)^{\frac{d^+}{2}}\left(\frac{\frac{1}{2} - \frac{it}{\pi}}{\frac{q-1}{2} - \frac{it}{\pi}} \right)^{\frac{d^-}{2}} e^{-\frac{\pi |\xi^+|^2}{4} \left( \frac{1 + \frac{2(q-1)it}{\pi}}{\frac{q-1}{2} + \frac{it}{\pi}}\right)}\,e^{-\frac{\pi |\xi^-|^2}{4} \left( \frac{1 - \frac{2(q-1)it}{\pi}}{\frac{q-1}{2} - \frac{it}{\pi}}\right)}\,\dt\\
& =  \lambda \, e^{-\frac{\pi (p-1)|\xi^+|^2}{2}}e^{-\frac{\pi (p-1)|\xi^-|^2}{2}}.
\end{align*}
Changing variables $t/\pi =s$, and clearing out constant factors we arrive at 
\begin{align}\label{20191007_16:05pm}
\begin{split}
& \int_{\R} (1 + 4s^2)^{-\frac{d(q-2)}{4}} \left(\frac{1 + 2is}{q-1 + 2is} \right)^{\frac{d^+}{2}} \left(\frac{1 - 2is}{q-1 - 2is} \right)^{\frac{d^-}{2}} e^{-\frac{\pi |\xi^+|^2}{2} \left( \frac{1 - (p-1)(q-1) + 2is(q-p)}{q-1 + 2is}\right)} \times \\
& \ \ \ \ \ \ \ \ \ \ \ \ \ \ \ \  \ \ \ \ \ \ \ \ \ \ \ \ \ \ \ \  \ \ \ \ \ \ \ \ \ \ \ \ \ \ \ \ \times e^{-\frac{\pi |\xi^-|^2}{2} \left( \frac{1 - (p-1)(q-1) - 2is(q-p)}{q-1 - 2is}\right)} \,\ds= \lambda\,,
\end{split}
\end{align}
where $\lambda$ is a constant (not necessarily the same as before). 

\smallskip

At this point observe that $d(q-2)/2 >1$ in our regime, which makes the integrand above absolutely integrable. Christ and Quilodr\'{a}n \cite{CQ} treated a similar integral as in \eqref{20191007_16:05pm}, with $d^- =0$ and $p\neq 2$, via contour integration in the complex plane. In our situation we may take advantage of the hiperbolicity to proceed with a purely real analysis. Regarding $|\xi^+|^2 = r^+$ and $|\xi^-|^2 = r^-$ as real variables in $[0,\infty)$ we may differentiate expression \eqref{20191007_16:05pm} $k$ times in the variable $r^+$ and $k$ times in the variable $r^-$, and then plug in $(r^+, r^-) = (0,0)$ to obtain (note that $1 - (p-1)(q-1) = -2p/d$)
\begin{equation}\label{20191007_16:40pm}
\int_{\R} \left( \frac{\frac{p^2}{d^2} + s^2(q-p)^2}{(q-1)^2 + 4s^2} \right)^k(1 + 4s^2)^{-\frac{d(q-2)}{4}} \left(\frac{1 + 2is}{q-1 + 2is} \right)^{\frac{d^+}{2}} \left(\frac{1 - 2is}{q-1 - 2is} \right)^{\frac{d^-}{2}}\,\ds = 0
\end{equation}
for all $k \in \N$. Let
$$A(s) = {\rm Re} \left[(1 + 4s^2)^{-\frac{d(q-2)}{4}} \left(\frac{1 + 2is}{q-1 + 2is} \right)^{\frac{d^+}{2}} \left(\frac{1 - 2is}{q-1 - 2is} \right)^{\frac{d^-}{2}}\right].$$
When $s$ is real we note that $A(s)$ is an even function that is continuous and non-identically zero (e.g. just observe the behaviour as $s \to \infty$). From \eqref{20191007_16:40pm} we have
$$\int_{0}^{\infty} \left( \frac{\frac{p^2}{d^2} + s^2(q-p)^2}{(q-1)^2 + 4s^2} \right)^k \, A(s) \, \ds = 0$$
for all $k \in \N$. Letting $B(s) = \left(\frac{\frac{p^2}{d^2} + s^2(q-p)^2}{(q-1)^2 + 4s^2} \right) A(s)$, by linearity we then have
\begin{equation}\label{20191007_20:46pm}
\int_{0}^{\infty} P\left( \frac{\frac{p^2}{d^2} + s^2(q-p)^2}{(q-1)^2 + 4s^2} \right) B(s) \,\ds = 0\,,
\end{equation}
where $P$ is any polynomial. 

\subsection{Changing variables} Note that in our regime we always have $q > 2(d+1)/d > p$, but as $p \to (2(d+1)/d)^-$ we have $q \to (2(d+1)/d)^+$. In every dimension $d \geq 2$ there exists an exponent $p_d \in (2, 2(d+1)/d)$ such that, for $q_d = (d+2)p_d'/d$, we have
\begin{equation}\label{20191007_20:55pm}
\frac{\frac{p_d}{d}}{q_d-p_d} = \frac{q_d-1}{2} = \kappa_d.
\end{equation}
This is in fact given by
$$p_d = \frac{ - d^2 + 8d + 4 + \sqrt{(d^2 - 8d -4)^2 + 32d^3}}{8d}.$$
If $p \neq p_d$, the change of variables
$$t = \varphi^{-1}(s) = \frac{\frac{p^2}{d^2} + s^2(q-p)^2}{(q-1)^2 + 4s^2} $$
is a bijection between $s \in (0,\infty)$ and $t \in \left(\frac{p^2/d^2}{(q-1)^2}, \frac{(q-p)^2}{4}\right)$ if $p  < p_d$, or $t \in \left( \frac{(q-p)^2}{4}, \frac{p^2/d^2}{(q-1)^2}\right)$ if $p  > p_d$. Hence, if $p \neq p_d$, equation \eqref{20191007_20:46pm} becomes
\begin{equation*}
\int_{\frac{p^2/d^2}{(q-1)^2}}^{\frac{(q-p)^2}{4}} P(t) \, B(\varphi(t)) \,\varphi'(t)\,\dt = 0.
\end{equation*}
Since $P$ is an arbitrary polynomial and the function $t \mapsto B(\varphi(t)) \,\varphi'(t)$ is continuous in the interior of the interval, not identically zero, and integrable in the interval, we get a contradiction by Weierstrass approximation.

\subsection{The case $p = p_d$} In this case we return to \eqref{20191007_16:05pm} and consider the diagonal $|\xi^+|^2 = |\xi^-|^2 = r \geq 0$. Recalling that $1 - (p-1)(q-1) = -2p/d$ and using \eqref{20191007_20:55pm} we obtain
\begin{equation*}
 \int_{\R} (1 + 4s^2)^{-\frac{d(q_d-2)}{4}} \left(\frac{1 + 2is}{q_d-1 + 2is} \right)^{\frac{d^+}{2}} \left(\frac{1 - 2is}{q_d-1 - 2is} \right)^{\frac{d^-}{2}} e^{-\frac{\pi r (q_d - p_d)}{2} \left( \frac{-\kappa_d + is}{\kappa_d + is} +  \frac{-\kappa_d - is}{\kappa_d - is}\right)} \,\ds= \lambda.
\end{equation*}
Differentiating $k$ times this function of $r$, and then plugging in $r=0$, we obtain
 \begin{equation*}
 \int_{\R} \left(\frac{-\kappa_d^2 + s^2}{\kappa_d^2 + s^2}\right)^k(1 + 4s^2)^{-\frac{d(q_d-2)}{4}} \left(\frac{1 + 2is}{q_d-1 + 2is} \right)^{\frac{d^+}{2}} \left(\frac{1 - 2is}{q_d-1 - 2is} \right)^{\frac{d^-}{2}} \,\ds= 0
\end{equation*}
for all $k \in \N$. We proceed as before by observing that the change of variables 
$$t = \gamma^{-1}(s) = \frac{-\kappa_d^2 + s^2}{\kappa_d^2 + s^2}$$
is a bijection between $s \in (0,\infty)$ and $t \in (-1,1)$, to reach a contradiction via Weierstrass approximation.

\section{Remarks on the saddle surface}

Let $S = \{(\xi_1,\xi_2, \tau) \in \R^3 \ : \ \tau = \xi_1^2 - \xi_2^2\}$. Here we address the $L^2(\R^2) \to L^4(\R^3)$ Strichartz inequality associated to $S$, presenting below a few important facts of this particular situation. Despite having the convolution structure, the quest for the sharp form of this inequality seems to be a subtle issue, as the previous methods for the paraboloid \cite{Fo07, HZ06} are not easily adaptable. This relates to the discussion in \cite[Appendix A]{DMPS}.

\medskip

In what follows $\xi = (\xi_1, \xi_2), \ \omega = (\omega_1, \omega_2), \ \eta = (\eta_1, \eta_2), \ \nu = (\nu_1, \nu_2)$ will be vectors in $\R^2$. For a Schwartz function $f:\R^2 \to \C$, we may use delta calculus to compute $\|Tf\|_{L^4(\R^{3})}$ as follows:
\begin{align*}
\|Tf\|^4_{L^4(\R^{3})} = \int_{(\R^2)^4} \delta_2(\xi + \omega - \eta - \nu)\,  \delta_1(Q(\xi) + Q(\omega) - Q(\eta) - Q(\nu))\, f(\xi) f(\omega) \overline{f(\eta)} \overline{f(\nu)} \,\dxi\,\domega\,\deta\,\dnu,
\end{align*}
where $\delta_d$ denotes the $d$-dimensional Dirac delta. We may consider an operator $K$ that acts on functions $F: \R^4 \to \C$ given by 
\begin{equation*}
KF(\eta, \nu) = \int_{(\R^2)^2} \delta_2(\xi + \omega - \eta - \nu)\,  \delta_1(Q(\xi) + Q(\omega) - Q(\eta) - Q(\nu))\, F(\xi, \omega) \,\dxi\,\domega.
\end{equation*}
We would then have
$$\|Tf\|^4_{L^4(\R^{3})} = \langle K(f\otimes f), (f\otimes f)\rangle\,,$$
where $\langle \cdot \, , \, \cdot \rangle$ denotes the inner product of $L^2(\R^4)$, and $(f\otimes f)(\eta, \nu):= f(\eta)f(\nu)$ is the usual tensor product of functions. We now make a few observations about this operator $K$.
\begin{proposition}
Let ${\bf 1}$ be the constant function equal to $1$ on $\R^4$. Then $K {\bf 1} \equiv \infty$.
\end{proposition}
\begin{proof}
Indeed, changing variables $\xi + \omega = \alpha$ and $\xi - \omega = \beta$ we have
\begin{align*}
\begin{split}
K {\bf 1}(\eta, \nu) & = \frac{1}{4} \int_{(\R^2)^2} \delta_2(\alpha - \eta - \nu)\,  \delta_1\left(\left(\frac{\alpha_1^2 + \beta_1^2}{2}\right) - \left(\frac{\alpha_2^2 + \beta_2^2}{2}\right) - Q(\eta) - Q(\nu)\right) \,{\rm d}\alpha\,{\rm d}\beta \\
& = \frac{1}{4} \int_{\R^2} \delta_1\left(\left(\frac{\beta_1^2 - \beta_2^2}{2}\right) + \left(\frac{(\eta_1 + \nu_1)^2 - (\eta_2 + \nu_2)^2}{2}\right) - Q(\eta) - Q(\nu)\right)\,{\rm d}\beta_1 \,{\rm d}\beta_2\\
& = \infty.
\end{split}
\end{align*}
\end{proof}
This is an important difference to the approach of Hundertmark and Zharnitsky \cite{HZ06} for the analogous estimate for the paraboloid. In that case, the corresponding operator $K$ applied to the constant function ${\bf 1}$ yielded a constant value, and this was crucial to establish the $L^2$-boundedness of the operator. In fact, we have the following negative result in this case.

\begin{proposition}\label{not bounded}
$K: L^2(\R^4) \to L^2(\R^4)$ is not bounded.
\end{proposition}
\begin{proof}
Let us compute $Kg$, where $g$ is the $L^2$-normalized Gaussian defined in \eqref{Gaussian}. Changing variables $\xi + \omega = \alpha$ and $\xi - \omega = \beta$ we have
\begin{align*}
Kg(\eta, \nu) & = \frac{1}{4} \int_{(\R^2)^2} \delta_2(\alpha - \eta - \nu)\,  \delta_1\left(\left(\frac{\alpha_1^2 + \beta_1^2}{2}\right) - \left(\frac{\alpha_2^2 + \beta_2^2}{2}\right) - Q(\eta) - Q(\nu)\right) \,e^{- \frac{\pi}{4} (|\alpha|^2 + |\beta|^2)}{\rm d}\alpha\,{\rm d}\beta\\
& =  e^{- \frac{\pi}{4} |\eta + \nu|^2} \int_{0}^{\infty} \int_{0}^{\infty}\delta_1\left(\left(\frac{\beta_1^2 - \beta_2^2}{2}\right) + \left(\frac{-(\eta_1 - \nu_1)^2 + (\eta_2 - \nu_2)^2}{2}\right)\right)\,e^{- \frac{\pi}{4}  |\beta|^2}\,{\rm d}\beta_1 \,{\rm d}\beta_2\\
& =  e^{- \frac{\pi}{4} |\eta + \nu|^2} \int_{0}^{\infty} \int_{0}^{\infty}\delta_1\left(u_1 - u_2 -(\eta_1 - \nu_1)^2 + (\eta_2 - \nu_2)^2\right)\,e^{- \frac{\pi}{4}  (u_1 + u_2)}\,\frac{{\rm d}u_1 \,{\rm d}u_2}{2 \sqrt{u_1 \,u_2}}.
\end{align*}
We now change variables by placing $u_1 - u_2 = x$ and $u_1 + u_2 = y$ to get
\begin{align*}
Kg(\eta, \nu) & = e^{- \frac{\pi}{4} |\eta + \nu|^2}\int_{-\infty}^{\infty} \delta_1\left(x -(\eta_1 - \nu_1)^2 + (\eta_2 - \nu_2)^2\right) \left(\int_{|x|}^{\infty} e^{- \frac{\pi}{4}y}\frac{{\rm d}y }{2 \sqrt{y^2 - x^2}} \right) {\rm d}x\\
& =  e^{- \frac{\pi}{4} |\eta + \nu|^2} \int_{-\infty}^{\infty} \delta_1\left(x -(\eta_1 - \nu_1)^2 + (\eta_2 - \nu_2)^2\right)  \frac{1}{2} K_0\big(\tfrac{\pi |x|}{4}\big) \,{\rm d}x\\
& = \frac12 e^{- \frac{\pi}{4} |\eta + \nu|^2}K_0\Big(\tfrac{\pi |(\eta_1 - \nu_1)^2 - (\eta_2 - \nu_2)^2|}{4}\Big),
\end{align*}
where $K_0$ denotes the modified Bessel function of second kind and order zero. Let us now show that $Kg \notin L^2(\R^4)$. By a change of variables $\eta + \nu = \alpha$ and $\eta - \nu = \beta$ we get 
\begin{align*}
\|Kg\|^2_{L^2(\R^4)} & = \frac{1}{4} \int_{\R^4}  e^{- \frac{\pi}{2} |\eta + \nu|^2}K_0\Big(\tfrac{\pi |(\eta_1 - \nu_1)^2 - (\eta_2 - \nu_2)^2|}{4}\Big)^2 \deta\,  \dnu\\
& = \frac{1}{16} \int_{\R^4}  e^{- \frac{\pi}{2} (\alpha_1^2 + \alpha_2^2)}\,K_0\Big(\tfrac{\pi |\beta_1^2 - \beta_2^2|}{4}\Big)^2\,{\rm d}\alpha \, {\rm d}\beta\\
& = \frac{1}{2}\int_0^{\infty} \int_0^{\infty}  K_0\Big(\tfrac{\pi |\beta_1^2 - \beta_2^2|}{4}\Big)^2\, {\rm d}\beta_1 \,{\rm d}\beta_2\\
& = \frac{1}{2}\int_0^{\infty} \int_0^{\infty} K_0\Big(\tfrac{\pi |u_1 - u_2|}{4}\Big)^2\, \frac{{\rm d}u_1 \,{\rm d}u_2}{4 \sqrt{u_1u_2}}\\
& = \frac{1}{2}\int_{-\infty}^{\infty} K_0\Big(\tfrac{\pi |x|}{4}\Big)^2 \left( \int_{|x|}^{\infty} \, \frac{ \,{\rm d}y}{4 \sqrt{y^2 - x^2}}\right) {\rm d}x \\
& = \infty.
\end{align*}
\end{proof}

Note that the inequality 
\begin{align}\label{20191007_23:43}
|\langle KF, F\rangle | \leq C \|F\|^2_{L^2(\R^4)}
\end{align}
cannot hold for all $F \in L^2(\R^4)$, otherwise a polarization argument would yield $K$ bounded on $L^2(\R^4)$, contradicting Proposition \ref{not bounded}.  Nevertheless, inequality \eqref{20191007_23:43} does hold restricted to the subclass of $L^2(\R^4)$ consisting of functions which are tensor products $f \otimes f$ with $f \in L^2(\R^2)$ (this is equivalent to the Fourier extension estimate). From the definition of the operator $K$ one can verify the following symmetry
$$KF(\eta_1, \eta_2, \nu_1, \nu_2) = KF(\nu_1, \eta_2, \eta_1, \nu_2).$$
Let $R$ be the reflection operator defined by $R(F)(\eta_1, \eta_2, \nu_1, \nu_2) := F(\nu_1, \eta_2, \eta_1, \nu_2)$. Let $E_1 = \{F \in C^{\infty}_c(\R^4) \ : \  F = R(F)\}$ and $E_2 = \{F \in C^{\infty}_c(\R^4) \ : \  F = -R(F)\}$. Note that $E_1$ and $E_2$ are orthogonal in $L^2(\R^4)$. Given $F \in C^{\infty}_c(\R^4)$ we can always write $F = F_1 + F_2$ with $F_1 \in E_1$ and $F_2 \in E_2$, namely by putting $F_1 = (F + R(F))/2$ and $F_2 = (F - R(F))/2$. One can also verify that 
$$ \langle KF, F\rangle = \langle KF_1, F_1\rangle.$$
It is tempting to then claim that an extremizer $f$ of our Fourier extension inequality should be such that $f \otimes f$ has the symmetry of $E_1$ but this is not, in principle, a valid claim. In fact, recall that \eqref{20191007_23:43} does not hold in all $L^2(\R^4)$, and if we try to symmetrize a tensor product $f \otimes f$ by changing it to the function 
$$\frac{1}{2} \big(f \otimes f + R(f \otimes f) \big)(\eta_1, \eta_2, \nu_1, \nu_2) = \frac{1}{2}\big( f(\eta_1,\eta_2)f(\nu_1, \nu_2) + f(\nu_1,\eta_2)f(\eta_1,\nu_2)\big)$$
we may leave the subclass of tensor products.

\end{document}